\documentclass[11pt,twoside,leqno]{amsart}
\usepackage{amsmath,amsfonts,amssymb,amsthm}

\theoremstyle{plain}
\newtheorem{theorem}{Theorem}
\newtheorem{proposition}[theorem]{Proposition}

\newtheorem{lemma}[theorem]{Lemma}

\theoremstyle{definition}
\newtheorem{definition}[theorem]{Definition}

\newtheorem{remark}[theorem]{Remark}
\newtheorem{remarks}[theorem]{Remarks}
\newtheorem{question}[theorem]{Question}
\newtheorem{problem}[theorem]{Problem}

\newcommand{\ve}{\varepsilon}

\newcommand{\bi}{{\mathbf i}}

\renewcommand{\H}{\mathbb{H}^1}
\newcommand{\hn}{\mathbb{H}^n}

\newcommand{\cB}{{\mathcal B}}

\newcommand{\cF}{{\mathcal F}}

\newcommand{\cH}{{\mathcal H}}
\newcommand{\cL}{{\mathcal L}}

\newcommand{\cS}{{\mathcal S}}

\newcommand{\field}[1]{\mathbb{#1}}
\newcommand{\N}{\field{N}}                  
\newcommand{\R}{\field{R}}                  
\newcommand{\C}{\field{C}}                  
\newcommand{\Sph}{\field{S}}
\newcommand{\V}{\field{V}}                  
\newcommand{\W}{\field{W}}                  

\newcommand{\ra}{\rightarrow}
\newcommand{\stm}{\setminus}

\newcommand{\deriv}[1]{\displaystyle{\frac{\partial}{\partial#1}}}

\DeclareMathOperator{\tanm}{Tan}

\DeclareMathOperator{\supp}{supp}
\DeclareMathOperator{\diam}{diam}

\newcommand{\restrict}{\begin{picture}(12,12)
                        \put(2,0){\line(1,0){8}}
                        \put(2,0){\line(0,1){8}}
                       \end{picture}}

\makeatletter               
\let\c@equation\c@theorem       
\makeatother                

\numberwithin{theorem}{section}

\title{Marstrand's density theorem in the Heisenberg group}
\author{Vasilis Chousionis}
\address{Department of Mathematics and Statistics \\ University of Helsinki \\ P. O. Box 68 \\ FI-00014, Finland}
\email{vasileios.chousionis@helsinki.fi}
\author{Jeremy T. Tyson}
\address{Department of Mathematics \\ University of Illinois \\ 1409
  West Green St. \\ Urbana, IL, 61801}
\email{tyson@illinois.edu}
\date{\today}
\thanks{VC funded by the Academy of Finland Grant SA 267047. JTT supported by NSF grant DMS-1201875.}

\begin{document}
\bibliographystyle{acm}

\begin{abstract}
We prove that if $\mu$ is a Radon measure on the Heisenberg group $\hn$ such that the density $\Theta^s(\mu,\cdot)$, computed with respect to the Kor\'anyi metric $d_H$, exists and is positive and finite on a set of positive $\mu$ measure, then $s$ is an integer. The proof relies on an analysis of uniformly distributed measures on $(\hn,d_H)$. We provide a number of examples of such measures, illustrating both the similarities and the striking differences of this sub-Riemannian setting from its Euclidean counterpart.
\end{abstract}

\maketitle

\section{Introduction and Notation}\label{sec:intro}

Let $\mu$ be a Radon measure on a metric space $(X,d)$. For $0\leq s < \infty$ the upper and lower $s$-densities of $\mu$ at $x \in X$ are defined respectively by
$$\Theta^{\ast s}(\mu,x)=\limsup_{r \ra 0} \frac{\mu(B(x,r))}{r^s} \qquad \text{ and } \qquad \Theta_{\ast}^{s}(\mu,x)=\liminf_{r \ra 0} \frac{\mu(B(x,r))}{r^s}.$$
In the case when $\Theta^{\ast s}(\mu,x)=\Theta_{\ast}^{s}(\mu,x)$, their common value is called the $s$-density of $\mu$ at $x$ and is denoted by $\Theta^s(\mu,x)$. Recall that in the case of $\R^n$ equipped with the usual Euclidean metric, the Lebesgue density theorem \cite[Corollary 2.14]{mat:geometry} asserts that whenever $A$ is  $\cL^n$-measurable, then $\Theta^n( \cL^n \restrict A, x)=2^n$ for $\cL^n$-a.e.\ $x \in A$ and $\Theta^n( \cL^n \restrict A, x)=0$ for $\cL^n$-a.e.\ $x \in \R^n \stm A$. Here  $\cL^n \restrict A$ denotes the restriction of the $n$-dimensional Lebesgue measure on $A$. Similar, but much weaker, results hold for Hausdorff and packing measures, see \cite[Chapter 6]{mat:geometry} and \cite{at:notes}.

Densities, and their connections to their underlying measures, have been studied extensively in the context of geometric measure theory  since the pioneering work of Besicovitch \cite{bes:fund} in the 1940's. One of the fundamental questions in this line of research is the following: \textit{assuming that $\mu$ is a Radon measure such that $\Theta^s(\mu,x)$ exists for $\mu$-a.e.\ $x$ what can be said about the properties of the measure $\mu$?} A major contribution due to Marstrand \cite{mardens} asserts that, in the Euclidean setting, if the $s$-density exists $\mu$-a.e then $s$ is an integer. In his seminal paper \cite{preiss:rectifiability}, Preiss showed that if the $m$-density of a Radon measure $\mu$ in $\R^n$, $m \in [0,n]$, exists $\mu$-a.e. then the measure $\mu$ is rectifiable, that is, there exist countably many $m$-dimensional Lipschitz graphs $M_i$ such that $\mu (\R^n \stm \cup_i M_i)=0$ and $\mu$ is absolutely continuous with respect to the $m$-dimensional Hausdorff measure $\cH^m$. For an informative and highly readable presentation of the theorem of Preiss we refer the reader to the monograph by De Lellis \cite{delellis}.

Obtaining analogues of the Marstrand and Preiss theorems for other metric spaces remains an interesting and highly non-trivial problem, see e.g.\ \cite[p.\ 112]{delellis}. Lorent \cite{lor1}, \cite{lor3}, \cite{lor2} considered metrics defined by polytope norms on finite-dimensional vector spaces. Our main goal in this paper is to prove Marstrand's theorem for the Heisenberg group equipped with a metric of sub-Riemannian type.

We now state the basic facts about the Heisenberg group needed in this paper. For an extensive treatment of the Heisenberg group from a variety of perspectives see e.g.\ \cite{mont:tour} or \cite{cdpt:survey}. The Heisenberg group $\hn$, identified with $\R^{2n+1}$, is a non-abelian Lie group whose group operation is given by
$$
x\cdot y=(x_1+y_1,\dots,x_{2n}+y_{2n},x_{2n+1}+y_{2n+1}+A(x',y')),
$$
where $x=(x',x_{2n+1})=(x_1,\ldots,x_{2n},x_{2n+1})$ and $y=(y',y_{2n+1})=(y_1,\ldots,y_{2n},y_{2n+1})$ and $A$ denotes the symplectic form on $\R^{2n}$ given by
\begin{equation}\label{A}
A(x',y')=-2\sum_{j=1}^n(x_jy_{j+n}-x_{j+n}y_j).
\end{equation}
For any $q \in \hn$ and $r >0$, let $\tau_q:\hn \ra \hn$ be the left translation
$$\tau_q(p)=q\cdot p$$
and let $\delta_r:\hn \ra \hn$ be the dilation
$$\delta_{r}(p)=(rp_1,\dots,rp_{2n},r^2p_{2n+1}).$$
These dilations are group homomorphisms. We denote by $e=(0,0) \in \R^{2n}\times\R$ the neutral element of $\hn$.

The Kor\'anyi metric $d_H$ on $\hn$ is defined by
$$
d_H(x,y)=\|x^{-1}\cdot y\|
$$
where
$$\|x\|=(\|x'\|^4_{\R^{2n}}+x^2_{2n+1})^{\frac{1}{4}}.$$
The metric is left invariant, that is $d_H(z\cdot x,z\cdot y)=d_H(x,y)$ for all $x,y,z\in\hn$, and the dilations satisfy
$d_H(\delta_r(x),\delta_r(y))=rd_H(x,y)$ for all $x,y\in\hn$ and $r>0$. The closed and open balls with respect to $d_H$ will
be denoted by $B(x,r)$ and $U(x,r)$ respectively. The $d_H$-diameter of a set $S \subset \hn$ will be denoted $\diam_H S$. Finally, the Euclidean metric on $\hn$ will be denoted by $d_E$.

We denote by $\mathcal{H}_H^s,s\geq 0,$ the $s$-dimensional Hausdorff measure obtained from the metric $d_H$, i.e. for $E \subset \hn$ and $\delta >0$, $\mathcal{H}_H^s (E)=\sup_{\delta>0} \mathcal{H}^s_{H,\delta} (E)$, where
$$\mathcal{H}^s_{H,\delta}(E)=\inf \left\{\sum_i  \diam_H(E_i)^s: E \subset \bigcup_i E_i,\,\diam_H (E_i)<\delta \right\}.$$
In the same manner the $s$-dimensional spherical Hausdorff measure for $E \subset \hn$ is defined as $\cS^s_H (E)=\sup_{\delta>0} \cS^s_{H,\delta} (E)$, where
$$\cS^s_{H,\delta}(E)=\inf \left\{\sum_i  r^s_i: E \subset \bigcup_i B(p_i,r_i),\,r_i \leq \delta,\,p_i \in \hn \right\}.$$
Translation invariance and dilation homogeneity of the Hausdorff measures follow as usual, therefore for $A \subset \hn, \ p \in \hn$, $s\ge 0$ and $r>0$,
$$\mathcal{H}^s_H (\tau_p(A))=\mathcal{H}^s_H(A) \ \text{and} \ \mathcal{H}^s_H (\delta_r(A))=r^s \mathcal{H}^s_H(A)$$  and the same relations hold for the spherical Hausdorff measures as well.  We will denote by $\dim_H(A)$ the Hausdorff dimension of a set $A \subset \hn$ with respect to the metric $d_H$, and by $\dim_E(A)$ the Hausdorff dimension with respect to the Euclidean metric in $\hn$. It is well known that the Hausdorff dimension of the metric space $(\hn,d_H)$ is equal to $Q=2n+2$.

Our main result reads as follows.

\begin{theorem}\label{marstrand}
Let $s>0$ and suppose that there exists a Radon measure $\mu$ on $(\hn,d_H)$ such that the density $\Theta^s(\mu,\cdot)$ exists and is positive and finite in a set of positive $\mu$ measure. Then $s$ is an integer.
\end{theorem}

Our proof does not follow the scheme of Marstrand's original proof. Instead we adopt and modify accordingly an argument due to Kirchheim and Preiss, who provided a different proof of Marstrand's theorem in \cite{kp:uniformly-distributed}. It is unknown to us if Marstrand's original proof---which has a strong Euclidean flavor---could be applied in the setting of the Heisenberg group. The proof of Kirchheim and Preiss relies on the geometric analysis of uniformly distributed measures. Such an analysis is of independent interest. Let us recall the definition.

\begin{definition}\label{uniform}
A Radon measure $\mu$ in $\hn$ is called uniformly distributed if
$$
\mu(B(x,r))=\mu(B(y,r))
$$
for all $x,y \in \supp(\mu)$ and $r>0$.
\end{definition}

A particular class of uniformly distributed measures are the $s$-uniform measures.

\begin{definition}\label{uniformmeas}
Given $s > 0$, a Radon measure $\mu$ in $\hn$ is called $s$-uniform if there exists some positive constant $c$ such that
$$
\mu(B(x,r))=c \, r^s
$$
for all $x \in \supp(\mu)$ and $r>0$.
\end{definition}

As in the proof of Kirchheim and Preiss an essential ingredient in the proof of Theorem \ref{marstrand} is the fact that the support of any uniformly distributed measure in $\hn$ is a real analytic variety in $\R^{2n+1}$. We show that at $\mu$-a.e.\ point where the $s$-density exists, there exist weak limits of blow-ups of $\mu$ which are $s$-uniform. In particular, these measures are uniformly distributed, hence their supports are real analytic varieties with Hausdorff dimension $s$. Using Lojasiewicz's structure theorem on analytic varieties and the fact that smooth submanifolds in $\hn$ have integer Hausdorff dimension (see Section \ref{sec:manif} for details), we conclude that $s$ is an integer.

In Section \ref{sec:examples} we discuss uniform and uniformly distributed measures in $\hn$ providing also several examples. The classification of uniform and uniformly distributed measures in $\R^n$ is a very difficult and largely unresolved problem. Marstrand's density theorem implies that there are no $s$-uniform measures for $s \notin \N$. Preiss in \cite{preiss:rectifiability} showed that for $m=1,2$, any $m$-uniform measure is $m$-flat, which means that it is of the form $c\, \cH^m \restrict V$, where $c$ is a positive constant and $V$ is an $m$-dimensional subspace. In the remarkable paper \cite{kp:besicovitch}, Kowalski and Preiss proved
that $\cH^3 \restrict C$ is $3$-uniform, where
$$C=\{x  \in \R^4:x_1^2+x_2^2+x_3^2=x_4^2\}$$
is the light cone in $\R^4$. Moreover they showed that every $(n-1)$-uniform measure in $\R^n$ is either $(n-1)$-flat or is a constant multiple of $\cH^{n-1}$ on some isometric copy of $C \times \R^{n-4}$. The classification of $m$-uniform measures in $\R^n$ remains open for $m \neq 1,2,n-1$. Recently Tolsa in \cite{tolsa-uniform} showed that $m$-uniform measures are uniformly $m$-rectifiable for any $m \leq n$ in $\R^n$. Uniform measures were also an essential tool in obtaining a new characterization of uniform rectifiability in \cite{cglt}. A characterization of uniformly distributed measures exists only for $\R$; this is due to Kirchheim and Preiss \cite{kp:uniformly-distributed}. In a future paper, we intend to return to the study of uniform and uniformly distributed measures and the Kowalski--Preiss theorem in $\hn$.

We emphasize that Marstrand's density problem is highly dependent on the metric and potentially sensitive to bi-Lipschitz deformation. It is an interesting open problem whether Theorem \ref{marstrand} holds for other metrics of sub-Riemannian type on $\hn$, or for metrics on more general Carnot groups. Among the features of the Kor\'anyi metric $d_H$ which enable the Kirchheim--Preiss argument to be transferred to this setting is the following analyticity criterion: {\it there exists a strictly decreasing real analytic function $g:[0,\infty)\to\R$ so that $g(t)$ decays to zero exponentially as $t\to\infty$ and $g\circ d_H$ is real analytic on $\R^{2n+1}\times\R^{2n+1}$}. (In the proof of Theorem \ref{marstrand} we use $g(t)=e^{-t^4}$.) We also employ Proposition \ref{manifold} on the integrality of the Hausdorff dimensions of smooth submanifolds of $\hn$ and the local uniformity of the spherical Hausdorff measure at transverse points. The latter results have been generalized to other Carnot groups, sometimes only for submanifolds of a specific type or of a specific dimension. The validity of such blowup estimates and related integral formulas for the spherical Hausdorff measure, for arbitrary submanifolds in arbitrary Carnot groups, remains a challenging problem of ongoing interest. This general program has been intensively investigated by Magnani and his collaborators, see the references cited in section~\ref{sec:manif}.

As mentioned earlier, only a few analogues of the Marstrand and Preiss theorems in other metric spaces are known. In particular, it is not known for which metrics on a Euclidean space $\R^N$ Marstrand's theorem is valid. The following result in this direction is an easy consequence of Theorem \ref{marstrand}. We equip $\R^{n+1}$ with the metric
\begin{equation}\label{heat-d}
d((x',x_{2n+1}),(y',y_{2n+1})) = (|x'-y'|^4+(x_{2n+1}-y_{2n+1})^2)^{1/4}
\end{equation}
where $x=(x',x_{2n+1})$ and $y=(y',y_{2n+1})$ are in $\R^n\times\R=\R^{n+1}$. Note that $d$ is translation invariant for the usual abelian group law on $\R^{n+1}$. To see that the function $d$ is a metric, it suffices to note that $(\R^{n+1},d)$ is isometric to the subgroup $(\W,d_H)$, where $\W = \{ (x_1,\ldots,x_n,0,\ldots,0,x_{2n+1}) \in \hn \, : \, x_1,\ldots,x_n,x_{2n+1} \in \R\}$. Denote by $B_d(x,r)$ the ball of radius $r$ and center $x$ in $(\R^{n+1},d)$.

\begin{theorem}\label{marstrand-corollary}
Let $\mu$ be a Radon measure on $(\R^{n+1},d)$ such that $\lim_{r\to 0} r^{-s} \mu(B_d(x,r))$ exists and is finite and positive in a set of positive $\mu$ measure. Then $s$ is an integer.
\end{theorem}

\begin{proof}
Define a measure $\tilde\mu$ on $\hn$ by $\tilde\mu(A) = \mu(A_1)$, where $A_1 = \{(x_1,\ldots,x_n,x_{2n+1}) \, : \, (x_1,\ldots,x_n,0,\ldots,0,x_{2n+1}) \in A \}$. Then $\tilde\mu$ is a Radon measure on $\hn$ and
$$
\tilde\mu(B_H(\tilde{x},r)) = \mu(B_d(x,r))
$$
for $\tilde{x} = (x_1,\ldots,x_n,0,\ldots,0,x_{2n+1}) \in \W$, where $x = (x_1,\ldots,x_n,x_{2n+1})$. The result follows from Theorem \ref{marstrand}.
\end{proof}

\

\paragraph{\bf Acknowledgements.} Thanks are due to Enrico Le Donne, Valentino Magnani and Pertti Mattila for conversations on the subject of this paper and for useful remarks. Research for this paper was carried out during a visit of JTT to the University of Helsinki and Aalto University. The hospitality of these institutions is appreciated.

\section{Differential geometry of submanifolds in the Heisenberg group}\label{sec:manif}

\subsection{Homogeneous subgroups}\label{subsec:homogeneous}

A basic class of uniform measures in the Heisenberg group $\hn$ consists of the natural volume measures on homogeneous subgroups. Recall that a subgroup $S \subset \hn$ is {\it homogeneous} if it is closed under the dilation semigroup, i.e., $\delta_r(x) \in S$ whenever $x \in S$ and $r>0$. Homogeneous subgroups of $\hn$ come in two flavors. A subgroup $\V$ is said to be a {\it horizontal homogeneous subgroup} if $\V = V\times(0) \subset \R^{2n}\times\R$, where $V$ is an isotropic subspace of $\R^{2n}$. (Recall that a subspace $V\subset\R^{2n}$ is said to be {\it isotropic} if the symplectic form $A$ defined in \eqref{A} vanishes on $V$.) The topological dimension of a horizontal homogeneous subgroup can be any value $k \in \{1,\ldots,n\}$. A subgroup $\W$ is said to be a {\it vertical homogeneous subgroup} if $\W = W\times\R \subset \R^{2n} \times \R$, where $W$ is any subspace of $\R^{2n}$. Vertical subgroups can have any topological dimension $k \in \{1,\ldots,2n\}$.

Horizontal and vertical homogeneous subgroups are both linear subspaces of the underlying Euclidean space $\R^{2n+1}$, however, their intrinsic metric structure as subsets of the Heisenberg group $\hn$ are quite diffferent. In particular, denoting by $k'$ the sub-Riemannian dimension of such a subgroup, we note that $k'=k$ for $k$-dimensional horizontal homogeneous subgroups and $k'=k+1$ for $k$-dimensional vertical homogeneous subgroups \cite[\S2.4]{mssc:intrinsic-rectifiability}. The natural volume measure on such a subgroup $\Sigma$ is the standard Lebesgue measure, which agrees up to a constant multiple with the restriction of the spherical Hausdorff measure, $\cS^{k'}_H\restrict\Sigma$, and also with the bi-invariant Haar measure \cite[Proposition 2.32]{mssc:intrinsic-rectifiability}. In connection with this paper the following result is of particular interest.

\begin{proposition}\label{homogeneous-subgroups}
For each homogeneous subgroup $\Sigma \subset \hn$ of sub-Riemannian dimension $k'$, $\cS^{k'}_H\restrict\Sigma$ is a $k'$-uniform measure.
\end{proposition}

\begin{proof}
Since $\Sigma$ is both homogeneous and a subgroup, $B(x,r) \cap \Sigma = (\tau_x\circ\delta_r)(B(e,1)\cap \Sigma)$ for each $x \in \Sigma$. Thus $\cS^{k'}_H(B(x,r)\cap\Sigma) = cr^{k'}$ where $c=\cS^{k'}_H(B(e,1)\cap\Sigma)$.
\end{proof}

In particular, $\cS^2_H\restrict\Sigma$ is $2$-uniform when $\Sigma$ is the vertical ($x_{2n+1}$-)axis. Note that the vertical axis is {\bf not} an intrinsically rectifiable subset of $\hn$ in the sense of \cite{mssc:intrinsic-rectifiability}. Thus the Preiss rectifiability theorem from \cite{preiss:rectifiability} fails to hold in $(\hn,d_H)$ when rectifiability is understood in the sense of
\cite{mssc:intrinsic-rectifiability}. A related observation was made by Lorent \cite[p.\  454]{lor3}.

\subsection{Geometry of submanifolds}\label{subsec:submanifolds}

The intrinsic geometry of submanifolds and more general subsets in sub-Riemannian spaces was advertised as a research program by Gromov in his pioneering work \cite{gro:cc} and has undergone intensive study since that time. In particular, Magnani has made a detailed analysis of the local structure of submanifolds of Carnot groups from the sub-Riemannian perspective, emphasizing blow-up estimates for volume measures, neglibility of characteristic points and associated area formulas for the spherical Hausdorff measure. This detailed program has been carried out in an ongoing series of papers, \cite{mag1}, \cite{mag2}, \cite{mag3}, \cite{mag4}, \cite{mag5}, \cite{mag6}, \cite{mag7}, \cite{mag8}. In this paper we only need to recall the relevant results in the setting of the Heisenberg group $\hn$ equipped with the Kor\'anyi metric.

Before stating these results we remind the reader that the sub-Riemannian differential geometric structure of $\hn$ derives from the fundamental notion of the {\it horizontal distribution} $H\hn$ which is a completely nonintegrable subbundle of the tangent bundle. The fiber $H_x\hn$, called the {\it horizontal tangent space} of $\hn $ at $x$, is the span of the values at $x$ of the left invariant vector fields
$$
X_j = \deriv{x_j} + 2x_{n+j} \deriv{x_{2n+1}} \quad \mbox{and} \quad X_{n+j} = \deriv{x_{n+j}} - 2x_{j} \deriv{x_{2n+1}}, \qquad \mbox{where $j=1,\ldots,n$.}
$$

\begin{proposition}[Magnani]\label{manifold}
Let $\Sigma$ be a $k$-dimensional $C^{1,1}$ submanifold in $\hn$. Then
\begin{enumerate}
\item $\dim_H (\Sigma) = k' \in \N$, where $k'$ is either $k$ or $k+1$,
\item  the measure $\cS^{k'}_H \restrict M$ is {\em asymptotically $k'$-uniform}, that is,
$$
\lim_{r \ra 0} \frac{\cS^{k'}_H(\Sigma \cap B(x,r))}{r^{k'}}
$$
exists for $\cS^{k'}_H$-a.e.\ $x \in \Sigma$.
\end{enumerate}
\end{proposition}

Part (1) of Proposition \ref{manifold} follows from the general formula for the Hausdorff dimensions of smooth submanifolds in (equiregular) sub-Riemannian manifolds given in \cite[\S0.6.B]{gro:cc}, or alternatively, as a consequence of part (2). The dimension $k'$ coincides with the {\it degree} $d(\Sigma)$ of $\Sigma$ as defined by Magnani \cite{mag:blowup}, \cite{mag4}. For submanifolds $\Sigma\subset\hn$ we have $k'=k+1$ whenever $n+1\le k\le 2n$, while if $1\le k\le n$ we may have either $k'=k$ or $k'=k+1$. In fact a $k$-dimensional submanifold $\Sigma \subset \hn$ has $k'=k$ if and only if $\Sigma$ is {\it horizontal}, that is, $T_x\Sigma$ is contained in $H_x\hn$ for all $x \in \Sigma$. The case $k'=k+1$ corresponds to {\it nonhorizontal} submanifolds, for which at least one tangent space $T_x\Sigma$ is transverse to the corresponding horizontal tangent space $H_x\hn$. Naturally, the distinction drawn here corresponds precisely to the distinction between horizontal and vertical homogeneous subgroups in subsection \ref{subsec:homogeneous}.

The asymptotic uniformity of $\cS^{k'}_H \restrict M$ holds at points of {\it maximal degree}, i.e., points $x \in \Sigma$ where the local degree $d_\Sigma(x)$ coincides with $d(\Sigma)$. For the definition of $d_\Sigma(x)$ in general Carnot groups, see \cite[p.\ 208]{mag4}. In the present setting the value of $d_\Sigma(x)$ for a $k$-dimensional submanifold $\Sigma$ is simply given by $k$ if $T_x\Sigma\subset H_x\hn$ and by $k+1$ otherwise. In fact, such asymptotic uniformity is stated in \cite[Theorem 1.1]{mag2} or \cite[Theorem 1.2]{mag4} for the volume measure on $\Sigma$ (relative to an auxiliary Riemannian metric) and holds at points of maximal degree. The $\cS^{k'}_H$ negligibility of points of lower degree can be observed in \cite[Corollary 1.2]{mag:blowup}, where it is stated for submanifolds in general Carnot groups of step two, or in \cite[Theorem 2.16]{mag3}. (In the latter reference the negligibility criterion is stated for maximally nonhorizontal submanifolds in general Carnot groups, however, in the Heisenberg group $\hn$ all submanifolds are either horizontal or maximally nonhorizontal, and all points in a horizontal submanifold automatically have maximal degree.) Conversion from the Riemannian volume measure to the spherical Hausdorff measure $\cS^{k'}_H$ is accomplished by means of an area formula relating these two measures.
See \cite[Theorem 1.2]{mag2} or \cite[(1.4)]{mag4}.

\section{Uniformly distributed measures and the proof of Marstrand's theorem}\label{sec:marstrand}

For a uniformly distributed measure $\mu$ in $\hn$ it is easy to see that
\begin{equation}\label{unididou}
\mu(B(x,r)) \leq \left(\frac{5r}{s}\right)^Q f_\mu(s)
\end{equation}
for every $x \in \hn$ and every $0<s<r<\infty$, where $f_\mu :(0,\infty) \ra (0,\infty)$ is defined by $f_\mu (s)=\mu(B(x,s))$ for any $x \in \supp(\mu)$. The proof of \eqref{unididou} is identical to the one of \cite[Lemma 1.1]{kp:uniformly-distributed}.

\begin{proposition}\label{realanal}
 Let $\mu$ be a uniformly distributed measure in $\hn$. Then $\supp(\mu)$ is a real analytic variety in $\R^{2n+1}$ and $\dim_H (\supp(\mu))$ is an integer.
\end{proposition}

\begin{proof}
If $\supp(\mu)=\hn$ then by \cite[Theorem 3.4]{mat:geometry} $\mu=c \, \cH_H^{2n+2}$, which is the Haar measure in $\hn$, hence trivially $\dim_H(\supp(\mu))=2n+2$. Therefore we can assume that $\supp(\mu) \neq \hn$. Let $x_0 \in \supp(\mu)$ and define
\begin{equation*}
F(x,s)=\int_{\R^{2n+1}} (\exp(-s\|x^{-1}\cdot z\|^4)-\exp(-s\|x_0^{-1}\cdot z\|^4)) \, d\mu(z),
\end{equation*}
for $x \in \hn$ and $s>0$. Using \cite[Theorem 1.15]{mat:geometry} we get,
\begin{equation}
\label{fub}
\begin{split}
\int_{\R^{2n+1}} \exp(-s\|x^{-1}\cdot z\|^4) d \mu (z)&= \int_0^\infty \mu(\{z : \exp(-s\|x^{-1}\cdot z\|^4) \geq t\})dt\\
&= \int_0^1 \mu \left( B\left(x, \sqrt[4]{\frac{-\log t}{s}}\right) \right)dt.
\end{split}
\end{equation}
Therefore,
\begin{equation}
\label{nondep}
\int_{\R^{2n+1}} \exp(-s\|x^{-1}\cdot z\|^4) d \mu (z)=\int_{\R^{2n+1}} \exp(-s\|y^{-1}\cdot z\|^4) d \mu (z)
\end{equation}
for all $x,y \in \supp(\mu)$ and all $s>0$, and hence
the function $F(x,s)$ is well defined as it does not depend on the choice of $x_0$. Using \eqref{fub} and \eqref{unididou} we also deduce that $F(x,s)$  is finite for all $x \in \hn$ and $s>0$, since
$$\int_{e^{-s}}^1 \mu \left( B\left(x, \sqrt[4]{\frac{-\log t}{s}}\right) \right)dt \leq \mu(B(x,1))<\infty$$
and
$$\int_0^{e^{-s}} \mu \left( B\left(x, \sqrt[4]{\frac{-\log t}{s}}\right) \right)dt \leq 5^Q f_\mu(1)\int_s^\infty \left(\frac{u}{s}\right)^{Q/4}e^{-u}du<\infty.$$

We will show that $x \in \supp(\mu)$ if and only if $F(x,s)=0$ for all $s>0$.  By \eqref{nondep} if $x \in \supp(\mu)$ then $F(x,s)=0$ for all $s>0$. Now let $x \notin \supp(\mu)$, it suffices to show that there exists $s>0$ such that $F(x,s) \neq 0$. Let $\ve \in (0,1)$ such that $B(x,\ve) \cap \supp(\mu)=0$. In that case, splitting the integral into annuli $B(x,(k+1)\ve)\setminus B(x,k \ve)$ and using \eqref{unididou} we get
\begin{equation*}
\begin{split}
\int_{\R^{2n+1}} \exp(-s\|x^{-1}\cdot z\|^4) d \mu (z)&\leq \sum_{k=0}^\infty  \exp(-s k^4 \ve^4)\, \mu(B(x, (k+1)\ve)) \\
&\leq 10^Q \, f_\mu(\ve/2) \sum_{k=0}^\infty (k+1)^Q\,\exp(-s k^4 \ve^4).
\end{split}
\end{equation*}
Noting that
\begin{equation*}
\begin{split}
\int_{\R^{2n+1}} \exp(-s\|x_0^{-1}\cdot z\|^4) d \mu (z) &\geq \int_{B(x_0, \ve/2)} \exp(-s\|x_0^{-1}\cdot z\|^4) d \mu (z) \\
&\geq \exp(-s (\ve/2)^4)\,f_\mu(\ve/2),
\end{split}
\end{equation*}
we deduce that
\begin{equation*}
\begin{split}
\lim_{s \ra \infty}& \frac{\int_{\R^{2n+1}} \exp(-s\|x^{-1}\cdot z\|^4) d \mu (z)}{\int_{\R^{2n+1}} \exp(-s\|x_0^{-1}\cdot z\|^4) d \mu (z)} \\
&\quad\quad\quad \leq \lim_{s \ra \infty} 10^Q \, \sum_{k=0}^\infty (k+1)^Q\,\exp(-s k^4 (\ve^4-\ve^4/16))=0.
\end{split}
\end{equation*}
Hence there exists some $s_x \in (0, \infty)$ such that for all $s>s_x$
$$\int_{\R^{2n+1}} \exp(-s\|x^{-1}\cdot z\|^4) d \mu (z)< \frac12 \int_{\R^{2n+1}} \exp(-s\|x_0^{-1}\cdot z\|^4) d \mu (z),$$
that is $F(x,s) \neq 0$ for $s>s_x$. Therefore we have shown that
$$
\supp(\mu)= \bigcap_{s>0} \{x \in \R^{2n+1}: F(x,s)=0\}.
$$
We now fix $s>0$. In order to show that $F(x):=F(x,s)$ is a real analytic function on $\R^{2n+1}$ it suffices to show that
$$F_1(x)=\int_{\R^{2n+1}} \exp(-s\|x^{-1}\cdot z\|^4) d \mu (z)$$
is real analytic. We define $\tilde{F_1}:\C^{2n+1} \ra \C$ for $\tilde{x}=(x_1,\dots,x_{2n+1}) \in \C^{2n+1}$ as
$$
\tilde{F_1}(\tilde{x})=\int_{\R^{2n+1}} \exp \left( -s \Big[ \big( \sum_{i=1}^{2n}(z_i-x_i)^2 \big)^2 +\big( z_{2n+1}-x_{2n+1}+2 \sum_{i=1}^n(x_iz_{i+n}-x_{i+n}z_i)\bigr) \Big]^2 \right) d \mu (z).
$$
It follows that $\tilde{F_1}$ is holomorphic on $\C^{2n+1}$, and as a consequence $F_1= \tilde{F_1}|_{\R^{2n+1}}$ is real analytic in $\R^{2n+1}$. Since the intersection of any family of analytic varieties is an analytic variety, see e.g.\ \cite{narasimhan-analytic}, we deduce that $\supp(\mu)$ is a real analytic variety in $\R^{2n+1}$.

It remains to show that $\dim_H(\supp(\mu))$ is an integer. We have shown  that the support of  $\mu$ in $(\hn,d_H)$ is an $m$-dimensional analytic variety of $\R^{2n+1}$ for some $m \in \N$. According to Lojasiewicz's Structure Theorem for real analytic varieties (see for example section 6.3 in \cite{kp:primer}), $\supp(\mu)$ can be written as the union of countably many analytic submanifolds of $\R^{2n+1}$ whose dimensions vary between $0$ and $m$. An application of part (1) of Proposition \ref{manifold} finishes the proof.
\end{proof}

Let $\mu$ be a Radon measure on a metric space $(X,d)$. A family of closed balls $\cF$ is said to be a {\it Vitali relation} for a set $S \subset X$ if for every $A \subset S$ there exists a disjoint, countable family of balls $\{B_i\}_{i \in I} \subset \cF$ such that
$$
\mu (A \setminus \cup_{i \in I}B_i)=0.
$$

\begin{theorem}
\label{Vitali}
Let $(X,d)$ and $\mu$ be as above. Let $G \subset X$ and $0<\alpha<\beta<\infty$ such that for every $x \in G$ there exists some $r_0(x)>0$ such that
$$
\mu(B(x, \alpha r)) \leq \beta \mu(B(x,r))$$
for all $0<r < r_0(x)$. Then the family of closed balls $\cF=\{B(x,r): x\in G, r < r_0(x)\}$
is a Vitali relation for $G$.
\end{theorem}

The proof of Theorem \ref{Vitali} follows after a few straightforward modifications in the proof of \cite[Theorem 1.6]{heinonen:book}. Using Theorem \ref{Vitali} we obtain the following version of the Lebesgue differentiation theorem.

\begin{theorem}\label{ldt}
Let $\mu$ be a Radon measure on a metric space $X$ and let $f$ be a nonnegative locally integrable function. If there exists a set $G \subset X$ as in Theorem \ref{Vitali} then
$$\lim_{r \ra 0} \frac{1}{\mu(B(x,r))} \int_{B(x,r)}f\, d\mu=f(x)$$
for $\mu$-a.e. $x \in G$.
\end{theorem}
We omit the proof of Theorem \ref{ldt} as it follows very closely the proof of \cite[Theorem 1.8]{heinonen:book}.

\begin{proof}[Proof of Theorem \ref{marstrand}]
By assumption the set $$G=\{x \in \hn: \Theta^s(\mu,x)\in (0, \infty)\}$$
is non-empty. It follows easily that for every $x \in G$ there exists some $r_0(x)>0$ such that $\mu(B(x,2r))\leq 2^{s+1} \mu(B(x,r))$ for all $r<r_0(x)$. In particular, by Theorem \ref{Vitali} the family of closed balls
$$\cF=\{B(x,r):x \in G, r <r_0(x)\}$$
is a Vitali relation in $G$. We consider the Borel measurable functions $d_k(x)=k^s\, \mu(B(x, \frac{1}{k}))$. By the theorems of Egorov and Lusin we conclude that there is a compact set $B \subset G$ with $\mu(B)>0$ and a continuous function $d:B \ra (0,\infty)$ such that $g_k$ converges to $d$ uniformly.
Since $\cF$ is a Vitali relation in $G$ and $B \subset G$ we can apply Theorem \ref{ldt} to the function $\chi_{B^c}$  and infer that
\begin{equation}\label{0dens}
\lim_{r \ra 0} \frac{\mu(B(x_0,r) \setminus B)}{\mu(B(x_0,r))}=0,
\end{equation}
for $\mu$-a.e.\ $x_0 \in B$.

We now pick some $x_0 \in B$ which satisfies \eqref{0dens} and define a sequence of measures $\{\nu_k\}_{k \in \N}$ by
$$
\nu_k(A)=k^s\mu(x_0 \cdot \delta_{\frac{1}{k}}(A)), \qquad \qquad A \subset \hn.
$$
Notice that for $N\in \N$,
\begin{equation*}
\begin{split}
\lim_{k \ra \infty} \nu_k(B(0, N))&=\lim_{k \ra \infty} k^s\mu(x_0 \cdot \delta_{\frac{1}{k}}(B(0, N)))\\
&=N^s\lim_{k \ra \infty} (k/N)^s\mu((B(x_0, N/k)))=N^s \, d(x_0)>0.
\end{split}
\end{equation*}
Therefore, since $C_0(\hn)$ is separable under the sup-norm, we can apply \cite[Theorem~1.23]{mat:geometry} in order to extract a subsequence $(\nu_{k_i})$ converging to a Radon measure $\nu$. In the sequel we use the standard facts that
$$
\nu(K) \geq \limsup_{i \ra \infty} \mu_{k_{i}}(K)
$$
and
$$
\nu(G) \leq \liminf_{i \ra \infty} \mu_{k_{i}}(G)
$$
for all compact $K$ and open $G$, see e.g.\ \cite[Theorem~1.24]{mat:geometry}. Since
\begin{equation*}\begin{split}
\nu(B(0, 1))& \geq \limsup_{i \ra \infty}\, (k_i)^s\mu(x_0 \cdot \delta_{\frac{1}{k_i}}(B(0, 1)))\\
&=\limsup_{i \ra \infty} \,(k_i)^s\mu\left(B\left(x_0, \frac{1}{k_i}\right)\right)= d(x_0)>0,
\end{split}\end{equation*}
we easily conclude that $\supp(\nu)$ is nonempty. Let $x \in \supp(\nu)$, $R>0$ and $\ve<R$. By \eqref{0dens} we have
\begin{equation*}
\begin{split}
\lim_{i \ra \infty} &\mu_{k_{i}}(B(0,\|x\|+2\ve) \setminus \delta_{k_i}(x_0^{-1}\cdot B))=\lim_{i \ra \infty} \,(k_i)^s\mu (x_0 \cdot \delta_{\frac{1}{k_i}}(B(0,\|x\|+2\ve) \setminus \delta_{k_i}(x_0^{-1}\cdot B))) \\
&=\lim_{i \ra \infty} \,(k_i)^s\mu\left(B\left(x_0, \frac{\|x\|+2\ve}{k_i}\right) \setminus B\right) \\
&=\lim_{i \ra \infty} \frac{\mu\left(B\left(x_0, \frac{\|x\|+2\ve}{k_i}\right) \setminus B\right)}{\mu\left(B\left(x_0, \frac{\|x\|+2\ve}{k_i}\right)\right)} \, (k_i)^s \mu\left(B\left(x_0, \frac{\|x\|+2\ve}{k_i}\right)\right)=0.
\end{split}
\end{equation*}
Using also the fact that
$$
\liminf_{i \ra \infty} \mu_{k_i} (B(x,\ve)) \geq \nu (B(x, \ve/2))>0,
$$
we conclude that there exist points $y_i \in B(x, \ve) \cap \delta_{k_i}(x_0^{-1}\cdot B)$. Observe that $x_0 \cdot \delta_{\frac{1}{k_i}}(y_i) \in B$. Therefore
\begin{equation*}
\begin{split}
\nu(B(x, R))& \geq \limsup_{i \ra \infty} \mu_{k_{i}}(B(x,R)) \geq \limsup_{i \ra \infty} \mu_{k_{i}}(B(y_i,R-\ve)) \\
&=\limsup_{i \ra \infty}\, (k_i)^s\mu(x_0 \cdot \delta_{\frac{1}{k_i}}(B(y_i, R-\ve))\\
&=\limsup_{i \ra \infty}\, (k_i)^s\mu\left( \left(B(x_0 \cdot \delta_{\frac{1}{k_i}}(y_i), \frac{R-\ve}{k_i}\right)\right) \geq (R-\ve)^s d(x_0),
\end{split}
\end{equation*}
where the final inequality follows because the functions $g_k$ converge uniformly to $d$.

In a similar manner we obtain that
$$
\nu (U(x, R+ \ve)) \leq d(x_0) (R+2\ve)^s.
$$
Letting $\ve \ra 0$ we conclude that $\nu$ is an $s$-uniform measure. Hence we deduce, see e.g.\ \cite[Theorem~2.4.3]{at:notes}, that $\dim_{H}(\supp(\nu))=s$. By Proposition \ref{realanal}, $\supp(\nu)$ is a real analytic variety in $\hn$ and in particular $\dim_{H}(\supp(\nu))$ and hence also $s$, is an integer.
\end{proof}

\begin{remark}\label{tangent}
Theorem \ref{marstrand} can also be established with the aid of tangent measures. Tangent measures, introduced by Preiss in \cite{preiss:rectifiability}, have subsequently become important tools in geometric measure theory. See \cite{mat:geometry} for an extensive treatment in Euclidean spaces, and \cite{mat:tan} or \cite{chousionis-mattila:hessio} for applications in the setting of metric groups with dilations, including the Heisenberg group.

Let $\mu$ be a Radon measure in $\hn$. The image $f_\#\mu$  under a map $f:\hn \to \hn$ is the measure
on $\hn$ defined by
$$f_\#\mu(A)=\mu\big(f^{-1}(A)\big)\ \text{for all }\ A \subset \hn.$$
For $a\in \hn$ and $r>0$, $T_{a,r}: \hn \to \hn$ is defined for all $p\in \hn$ by
$$T_{a,r}(p)= \delta_{1/r}(a^{-1}\cdot p).$$

\begin{definition}\label{tanm}
Let $\mu$ be a Radon measure on $\hn$.
We say that $\nu$ is a \emph{tangent measure} of $\mu$ at $a\in \hn$ if $\nu$ is a Radon measure on
$\hn$ with $\nu(\hn)>0$ and there are positive numbers $c_i$ and
$r_i$, $i=1,2,\dots$, such that $r_i\to 0$ and
$$c_iT_{a,r_i\#}\mu \to \nu\ \text{weakly as}\ i\to\infty.$$
We denote by $\tanm(\mu,a)$ the set of all tangent measures of $\mu$ at a.
\end{definition}

Let $\mu$ as in the assumptions of Theorem \ref{marstrand}. With the help of Theorems \ref{Vitali} and \ref{ldt} we can reproduce the argument in Theorem \cite[Lemma~14.7.1]{mat:geometry} in order to show that if $$A=\{x \in \hn: \Theta^s(\mu,x)\mbox{ is positive and finite}\},$$
then for $\mu$-a.e. $x \in A$, every $\nu \in \tanm (\mu,x)$ is $s$-uniform. Hence we deduce, see e.g.\ \cite[Theorem~2.4.3]{at:notes}, that $\dim_{H}(\supp(\nu))=s$. By Proposition \ref{realanal}, $\supp(\nu)$ is a real analytic variety in $\hn$ and in particular $\dim_{H}(\supp(\nu))$ and hence also $s$, is an integer.
\end{remark}

The following theorem is a Heisenberg adaptation of \cite[Corollary 1.6]{kp:uniformly-distributed}.

\begin{theorem}\label{algebraic}
Let $\mu$ be a uniformly distributed measure in $(\hn,d_H)$ with bounded support. Then $\supp(\mu)$ is an algebraic variety.
\end{theorem}

\begin{proof}
Since $\supp(\mu)$ is bounded, the function $F(x,s)$ from Proposition \ref{realanal} admits the expansion
$$
F(x,s) = \sum_{j=0}^\infty \frac{(-s)^j}{j!} \int_{\R^{2n+1}} \left( \|x^{-1}\cdot z\|_H^{4j} - \|x_0^{-1}\cdot z\|_H^{4j} \right) \, d\mu(z).
$$
It easily follows that $F(x,s)=0$ for every $s>0$ if and only if the functions
$$
P_j(x) = \int_{\R^{2n+1}} \left( \|x^{-1}\cdot z\|_H^{4j} - \|x_0^{-1}\cdot z\|_H^{4j} \right) \, d\mu(z).
$$
vanish for every $j=1,2,\ldots$. Each $P_j$ is a polynomial in the coordinates of the point $x \in \hn$. As in the proof of Corollary 1.6 of \cite{kp:uniformly-distributed}, an appeal to Hilbert's theorem for polynomials over a Noetherian ring implies that $\supp(\mu)$ coincides with the simultaneous vanishing set of finitely many of the polynomials $P_j$, whence $\supp(\mu)$ is an algebraic variety.
\end{proof}

\section{Examples and discussion}\label{sec:examples}

In this section we exhibit uniform or uniformly distributed measures in the Heisenberg group $\hn$ equipped with the Kor\'anyi metric $d_H$. New phenomena arise which lack any Euclidean counterpart. Numerous questions remain; we indicate several of these in the course of our discussion.

\subsection{Uniform measures}\label{subsec:uniform}

We have already remarked (see section \ref{subsec:homogeneous}) that the volume measure on any homogeneous subgroup of $\hn$ is uniform. In particular, vertical hypersurfaces of $\hn$ support $(Q-1)$-uniform measures. Recall that $Q=2n+2$ is the Hausdorff dimension of $(\hn,d_H)$. In Euclidean space $\R^n$, the classification of uniform measures supported on hypersurfaces (possibly with singularities) is due to Kowalski and Preiss \cite{kp:besicovitch}: any such measure is a constant multiple of the Hausdorff measure $\cH^{n-1}$ supported on either a hyperplane or an isometric image of $C\times\R^{n-4}$ where $C$ denotes the light cone
$$
C := \{ x \in \R^4 \, | \, x_4^2 = x_1^2 + x_2^2 + x_3^2 \}.
$$

The Kowalski--Preiss light cone example can be isometrically embedded into Heisenberg groups of sufficiently large dimension, yielding new examples of uniform measures supported on submanifolds. Note that due to the geometry of the Heisenberg group we do {\bf not} obtain that the codimension one algebraic variety
$$
\{ (x_1,\ldots,x_n,x_{n+1},\ldots,x_{2n},x_{2n+1} ) \in \hn  \, | \, (x_1,x_2,x_3,x_4) \in C \},
$$
supports a $(Q-1)$-uniform measure. We do however obtain the following

\begin{proposition}\label{light-cones}
Let $n\ge 4$.

(i) The set
$$
C_H := \{ (x_1,\ldots,x_k,0,\ldots,0,0) \in \hn  \, | \, (x_1,x_2,x_3,x_4) \in C, 4\le k \le n \},
$$
or any image of $C_H$ by an isometry of $(\hn,d_H)$, supports a $(k-1)$-uniform measure.

(ii) The set
\begin{equation}\label{tildeCH}
\tilde{C_H} := \left\{ (x_1,\ldots,x_k,0,\ldots,0,x_{2n+1} ) \in \hn  \, | \, (x_1,x_2,x_3,x_4) \in C, 4\le k \le n, x_{2n+1} \in \R \right\},
\end{equation}
or any image of $C_H$ by an isometry of $(\hn,d_H)$, supports a $(k+1)$-uniform measure.
\end{proposition}

It is well known that the isometries of $\hn$ are generated by left translations, by the `reflection' $(x_1,\ldots,x_{2n+1}) \mapsto (-x_1,\ldots,-x_n,x_{n+1},\ldots,x_{2n},-x_{2n+1})$, and by the rotations $(x',x_{2n+1}) \mapsto (Ax',x_{2n+1})$, where $A \in U(n)$ acts on the point $(x_1+\bi x_{n+1},\ldots,x_n+\bi x_{2n}) \in \C^n$, $x'=(x_1,\ldots,x_{2n})$.
Moreover, the similarities of $\hn$ are generated by the isometries of $\hn$ together with the dilations $\delta_r$, $r>0$.

\begin{proof}
To prove (i), fix $k$, $4\le k\le n$, as in the statement of the proposition, and consider the horizontal homogeneous subgroup $\V = \{ (x_1,\ldots,x_k,0,\ldots,0) \in \hn \, : \, x_1,\ldots,x_k \in \R \}$. The restriction of $d_H$ to $\V$ coincides with the restriction of the Euclidean metric of $\R^{2n+1}$ to $\V$. Since $C_H \subset \V$ we see that for $x \in C_H$ and $r>0$,
\begin{equation}\label{equality-of-balls}
B_H(x,r) \cap C_H = B_E(x,r) \cap C_H
\end{equation}
where $B_E(x,r)$ denotes the Euclidean ball in $\R^{k}$ with center $x$ and radius $r$. By the result of Kowalski--Preiss, the set $C_H$, equipped with the Euclidean metric, supports a $(k-1)$-uniform measure. By \eqref{equality-of-balls} this measure is also $(k-1)$-uniform for the Kor\'anyi metric restricted to $C_H$.

Part (ii) follows from part (i) and the following lemma.
\end{proof}

\begin{lemma}
Let $S \subset \R^n$ support an $m$-uniform measure $\mu$ for some $m \in \{0,1,\ldots,n\}$. Then $(S\times\R,d)$ supports an $(m+2)$-uniform measure, where $d$ denotes the metric on $\R^{n+1}$ given in \eqref{heat-d}.
\end{lemma}

\begin{proof}
Let $c>0$ be such that $\mu(B(x,r))=cr^m$ for all $x \in S$ and $r>0$. Consider the measure $\nu=\mu\otimes\cL^1$ on $S\times\R$. For $x=(x',x_{2n+1}) \in S \times \R$ and $r>0$ we use the Fubini theorem to derive
\begin{equation*}\begin{split}
\nu(B_d(x,r)\cap(S\times\R)) &= \int_{x_{2n+1}-r^2}^{x_{2n+1}+r^2} \mu(B(x',\sqrt[4]{r^4-(x_{2n+1}-y_{2n+1})^2})\cap S) \, dy_{2n+1} \\
&= c\int_{x_{2n+1}-r^2}^{x_{2n+1}+r^2} (r^4-(x_{2n+1}-t)^2)^{m/4} \, dt \\
&= c\int_{-r^2}^{r^2} (r^4-t^2)^{m/4} \, dt = cr^{m+2} \int_{-1}^{1} (1-\tau^2)^{m/4} \, d\tau.
\end{split}\end{equation*}
Hence $\nu$ is an $(m+2)$-uniform measure on $S\times\R$.
\end{proof}

\begin{remarks}
In \cite[\S 3]{kp:besicovitch}, Kowalski and Preiss draw further conclusions in the Euclidean setting. For instance, they prove that $\cH^{m+1} \restrict (M\times\R)$ is $(m+1)$-uniform if and only if $\cH^m \restrict M$ is $m$-uniform. The proof of the converse direction uses the equality of $\cH^{m+1} \restrict (M\times\R)$ with the product measure $\cH^m\restrict M \otimes \cL^1$ for countably $m$-rectifiable sets $M \subset \R^n$. In our setting, when the metric on $\R^{n+1}$ is not the standard Euclidean metric, we do not know whether such equality holds.
\end{remarks}

\begin{problem}
Does there exist any $C^1$ hypersurface (or more generally, algebraic variety) in $(\hn,d_H)$, other than vertical hyperplanes, which supports a uniform measure?
\end{problem}

\subsection{Uniformly distributed measures}\label{subsec:unif-distributed}

Kirchheim and Preiss \cite[Section 2]{kp:uniformly-distributed} characterized uniformly distributed measures in $\R$ and gave examples of such measures in $\R^2$. In this section we present examples in the first Heisenberg group $(\H,d_H)$.

Proposition \ref{algebraic} ensures that bounded supports of uniformly distributed measures on $\hn$ are algebraic varieties. Bounded supports of Euclidean uniformly distributed measures are contained in spheres, see \cite{chr:uniform} or  \cite[Proposition~1.7]{kp:uniformly-distributed}. We are currently unable to obtain analogous simple conclusions in the Heisenberg setting. This fact complicates attempts to characterize Heisenberg uniformly distributed measures.

We first observe (cf.\ \cite[Remark 2.5]{kp:uniformly-distributed}) that a locally finite measure is uniformly distributed provided it is invariant under a group of isometries acting transitively on the support. More precisely, if a locally finite Borel measure $\mu$ has the property that for each $x,y\in\supp(\mu)$ there exists an isometry $\Phi$ of $(\H,d_H)$ such that $\Phi(x)=y$ and $\Phi_\#\mu=\mu$, then $\mu$ is uniformly distributed. Following Kirchheim and Preiss, let us call such measures {\it homogeneous}. Recall that the isometries of $\H$ are generated by left translations, rotations about the $x_3$-axis, and the ``horizontal reflection'' $\rho$ defined by $\rho(x_1,x_2,x_3)=(x_1,-x_2,-x_3)$, and that the similarities of $\H$ are generated by the isometries and the dilations $\delta_r(x_1,x_2,x_3)=(rx_1,rx_2,r^2x_3)$.

First we consider counting measure on finite sets. In the plane, uniformly distributed counting measures with finite support are supported on either the vertices of a regular polygon, or two regular $m$-gons lying on a common circle \cite[Proposition 2.4]{kp:uniformly-distributed}. We adapt this example to $\H$ as follows. The proof of the result consists of applying the aforementioned homogeneity criterion.

\begin{proposition}
The restriction of counting measure to each of the following finite sets $A$, or its image under a similarity of $\H$, is uniformly distributed.
\begin{itemize}
\item[(i)] $A$ consists of the vertices of a regular polygon lying on the circle $\Sph^1 \times \{0\} \subset \H$, or two regular $m$-gons lying on $\Sph^1 \times \{0\}$.
\item[(ii)] For any given $\delta>0$, $A$ consists of the vertices of a regular $m$-gon lying on the circle $\Sph^1 \times \{\delta\}$ together with the vertices of a regular $m$-gon lying on the circle $\Sph^1 \times \{-\delta\}$.
\end{itemize}
\end{proposition}

A set $A$ in a metric space $(X,d)$ is called {\it equilateral} if $d(x,y)$ is constant for all $x,y\in A$, $x\ne y$. It is clear that counting measure is uniformly distributed on each equilateral set. We investigate the equilateral subsets of $(\H,d_H)$. In contrast with the Euclidean case, there exist equilateral sets on which the isometries do not act transitively. It is interesting to observe that there exist non-homogeneous uniformly distributed measures in $(\H,d_H)$, cf.\ the question on p.\ 159 in \cite{kp:uniformly-distributed}.

Every one or two point subset $A \subset \H$ is trivially equilateral. The following proposition characterizes equilateral triangles in $(\H,d_H)$. Such triangles fall into three distinct classes: (i) two vertices lie on a vertical line, (ii) two vertices lie on a horizontal line, and (iii) no two vertices lie on either a horizontal or a vertical line.

\begin{proposition}\label{equilateral-triangles}
The following sets $A \subset \H$, or their images under similarities of $\H$, are the only equilateral sets for the Kor\'anyi metric $d_H$.
\begin{itemize}
\item[(i)] $A = \{(0,0,1),(0,0,-1),((3/4)^{1/4},0,0)\}$.
\item[(ii)] $A = \{(1,0,0),(-1,0,0),(r\cos\theta,r\sin\theta,t)$, where
$$
r=r(\theta) = \left( 2 \sin\theta \sqrt{5+\sin^2\theta} - 2\sin^2\theta - 1 \right)^{1/2}
$$
and
$$
t=t(\theta) = (\cot\theta)(r^2+1)$$
for some $\theta \in [\arcsin(1/4),\pi-\arcsin(1/4)]$.
\item[(iii)] $A = \{(-x_0,\frac{\sqrt3}2,0),(-x_0,-\frac{\sqrt3}2,0),(r\cos\theta,r\sin\theta,t)\}$, where $x_0>0$,
\begin{equation}\label{t-equation}
t=t(r,\theta) = -(\tan\theta)(r^2+x_0^2+\tfrac34)
\end{equation}
and $r$ and $\theta$ are related by the implicit equation
\begin{equation}\label{xy-equation}
3(3+4x_0^2-r^2)\cos^2\theta = \left( \frac34 + \left|x_0 + re^{\bi\theta}\right|^2 \right)^2.
\end{equation}
\end{itemize}
\end{proposition}

\begin{proof}
The fact that the sets in parts (i) and (ii) are equilateral is verified by a direct computation which we omit. Note that any two points of $\H$ contained in a vertical line can be mapped by a similarity of $\H$ onto the points $(0,0,1)$ and $(0,0,-1)$. Similarly, any two points of $\H$ contained in a horizontal line can be mapped by a similarity onto the points $(1,0,0)$ and $(-1,0,0)$.

To finish the proof, we first confirm that any two points of $\H$ which do not both lie on a horizontal line or both lie on a vertical line can be mapped by a similarity of $\H$ onto the points $u=(-x_0,\frac{\sqrt3}2,0)$ and $v=(-x_0,-\frac{\sqrt3}2,0)$ for some $x_0>0$.

By applying a left translation we may assume that one of the two points is the origin $e=(0,0,0)$, while the other is of the form $y=(y_1,y_2,y_3)$ with both $y_1^2+y_2^2>0$ and $y_3 \ne 0$. We show that there exists $x_0>0$ and a similarity of $\H$ which maps $u$ and $v$ onto $e$ and $y$. First, left translate $u$ and $v$ by the inverse of $u$. This sends $v$ to $u^{-1}\cdot v = (0,-\sqrt3,2\sqrt3x_0)$. Apply a suitable rotation about the vertical axis and dilate by $\rho/\sqrt3>0$ to send the latter point to
$$
(\rho\cos\varphi,\rho\sin\varphi,(\tfrac2{\sqrt3})\rho^2x_0).
$$
We seek a solution in the variables $\rho$, $\varphi$ and $x_0$ to the equations $y_1 = \rho\cos\varphi$, $y_2 = \rho\sin\varphi$ and $y_3 = (\tfrac2{\sqrt3})\rho^2x_0$. Since $y_1^2+y_2^2>0$ we may choose $\rho=\sqrt{y_1^2+y_2^2}>0$ and $\varphi$ to satisfy the first two equations. Then selecting
$$
x_0 = \frac{\sqrt3}2 \frac{y_3}{\rho^2} = \frac{\sqrt3}2 \frac{y_3}{y_1^2+y_2^2}
$$
finishes this step of the proof.

Observe that $d_H(u,v) = R:= (9+12x_0^2)^{1/4}$. It now suffices to find all points $x=(x_1,x_2,x_3) \in \H$ such that $d_H(u,x)=d_H(v,x)=R$. An extensive but elementary calculation shows that all such points can be expressed in the form shown in the statement of the proposition. This completes the proof.
\end{proof}

\begin{remark}
We chose the normal form $u=(-x_0,\frac{\sqrt3}2,0)$ and $v=(-x_0,-\frac{\sqrt3}2,0)$ deliberately. When $x_0=\tfrac12$ we obtain two cube roots of unity $-\tfrac12\pm\bi\frac{\sqrt3}2$. In this case the choice $r=1$, $\theta=0$ and $t=0$ is allowed in equations \eqref{xy-equation} and \eqref{t-equation}. Indeed, the vertices of the standard equilateral triangle in the horizontal plane $\{x_3=0\}$ of $\H$ remains an equilateral set in $(\H,d_H)$, since this set is homogeneous in the sense of Kirchheim and Preiss.
\end{remark}

A characterization of four point equilateral sets is likely tractable, however, we decline to carry out such an analysis here. We only remark that four point equilateral sets in $\H$ do exist. For instance, consider the set $\{w,x,y,z\} \subset \H$ where $\{x,y,z\}$ denote the vertices of the standard equilateral triangle in the horizontal plane and $w=(0,0,w_3)$ for a suitable choice of $w_3>0$. Note that when $w_3=0$ the common distance from $w$ to any of the points $x$, $y$ or $z$ is strictly smaller than the common mutual distance between $x$, $y$ and $z$, while when $w_3 \to +\infty$ the common distance from $w$ to $x$, $y$ and $z$ also tends to $+\infty$. By continuity, there exists a choice of $w_3>0$ so that $\{w,x,y,z\}$ is equilateral.

\begin{question}
Are there any equilateral five point subsets of $(\H,d_H)$?
\end{question}

Infinite discrete subsets of $(\H,d_H)$ on which counting measure is uniformly distributed include the integer points in a horizontal or vertical line, or any similarity image of such a set. We do not know whether there are any other examples.

We turn to measures supported on curves. Of course, length measure along a horizontal line is a $1$-uniform measure. The following proposition gives additional examples which are supported on nonhorizontal curves, either bounded or unbounded. We conjecture that there are no uniformly distributed measures in the first Heisenberg group which are supported on bounded horizontal curves.

\begin{proposition}
The restriction of $\cS^2_H$ to each of the following sets $A$, or its image under a similarity of $\H$, is uniformly distributed.
\begin{itemize}
\item[(i)] The unit circle $A = \Sph^1 \times (0)$.
\item[(ii)] For each $a<b$, the set $A = \Sph^1 \times \{a,b\}$.
\item[(iii)] The set $A$ consisting of the vertical lines passing through the vertices of a regular polygon lying on the circle $\Sph^1 \times \{0\}$, or the vertical lines passing through the vertices of two regular $m$-gons lying on $\Sph^1 \times \{0\}$.
\end{itemize}
\end{proposition}

Finally, we discuss measures supported on surfaces. The volume measure $\cS^3_H$ restricted to a vertical hyperplane is $3$-uniform. We give an additional example.

\begin{proposition}
The restriction of $\cS^3_{H}$ to the right circular cylinder
$$
A = \{(x_1,x_2,x_3) \in \H \, | \, x_1^2+x_2^2=1\},
$$
or its image under any similarity of $\H$, is uniformly distributed.
\end{proposition}

\begin{question}\label{Q0}
Are there other examples of uniformly distributed measures in $\H$ supported on $C^1$ surfaces?
For instance, are any of the following sets in $\H$ the support of a uniformly distributed measure?
\begin{itemize}
\item[(i)] The Kor\'anyi unit sphere $\{x \in \hn : |x|_H = 1 \}$.
\item[(ii)] Pansu's bubble set $\cB$ (see \cite{cdpt:survey}).
\item[(iii)] Closed horizontal lifts of the figure $8$ curve $\{(x_1,x_2) \in \R^2 \, | \, (x_1\pm 1)^2+x_2^2 = 1 \}$.
\end{itemize}
\end{question}

We anticipate that the answer to part (i) of the preceding question is no. The rationale for parts (ii) and (iii) comes from the fact that Euclidean spheres, which are examples of supports of Euclidean uniformly distributed measures in all dimensions, are surfaces of constant mean curvature. As is well known, Pansu's bubble set $\cB$ is a surface of constant {\it horizontal mean curvature} while the curves in part (iii) also have constant horizontal curvature.

\begin{question}\label{Q1}
Are there any uniformly distributed measures in $(\H,d_H)$ with compact and infinite support whose support is intrinsically rectifiable in the sense of \cite{mssc:intrinsic-rectifiability}?
\end{question}

In this paper we have focused exclusively on the Kor\'anyi metric $d_H$ on $\hn$. Our rationale for this choice was described in the introduction. However, the following interesting question remains.

\begin{question}\label{Q3}
Let $d_{cc}$ denote the Carnot--Carath\'eodory metric on $\hn$. What can be said about uniform or uniformly distributed measures in $(\hn,d_{cc})$?
\end{question}

\bibstyle{acm}
\bibliography{marstrand}
\end{document}